\documentclass{article}
\usepackage{graphicx,xcolor,textpos}
\usepackage{helvet}
\usepackage{amssymb,latexsym,amsmath,epsfig,amsthm,hyperref,tipa,comment,enumitem}
\usepackage{xspace}  

\newtheorem{explicitgraham}{Theorem}
\newtheorem{actualsninftybound}[explicitgraham]{Theorem}
\newtheorem{finalbound}[explicitgraham]{Theorem}
\newtheorem{primesarehard}[explicitgraham]{Theorem}

\newtheorem{cr1}{Lemma}
\newtheorem{cr2}[cr1]{Lemma}
\newtheorem{sninftybound}[cr1]{Lemma}
\newtheorem{sninftylemma}[cr1]{Lemma}
\newtheorem{tenenbaumyokota}[cr1]{Lemma}
\newtheorem{ronald}[cr1]{Lemma}
\newtheorem{arbsmall}[cr1]{Lemma}
\newtheorem{halfandone}[cr1]{Lemma}
\newtheorem{allhalfs}[cr1]{Lemma}
\newtheorem{completesequences}[cr1]{Lemma}
\newtheorem{sumofstuff}[cr1]{Lemma}
\newtheorem{smoothinarith}[cr1]{Lemma}
\newtheorem{existenceofd}[cr1]{Lemma}

\newcommand{\mfree}{$M$-free\xspace}
\newcommand{\mlarge}{$m$-large\xspace}

\newcommand{\free}[1]{#1-free\xspace}

\newcommand{\nalpham}{n_{\alpha, m}\xspace}
\newcommand{\nalphaone}{n_{\alpha, 1}\xspace}
\newcommand{\nalpha}{n_{\alpha}\xspace}
\newcommand{\n}[1]{n_{#1}\xspace}

\newcommand{\setA}[1]{A(#1)\xspace}
\newcommand{\setB}[1]{B(#1)\xspace}

\begin{document}
\vspace*{-2cm}
\Large
 \begin{center}
Partitions with prescribed sum of reciprocals: \\ 
asymptotic bounds

\hspace{10pt}

\large
Wouter van Doorn \\

\hspace{10pt}

\end{center}

\hspace{10pt}

\normalsize

\vspace{-10pt}

\centerline{\bf Abstract}

\noindent
In \mbox{\cite{gr}} Graham proved that every positive integer $n \ge 78$ can be written as a sum of distinct positive integers $a_1, a_2, \ldots, a_r$ for which $\frac{1}{a_1} + \frac{1}{a_2} + \ldots + \frac{1}{a_r}$ is equal to $1$. In the same paper he managed to further generalize this, and showed that for all positive rationals $\alpha$ and all positive integers $m$, there exists an $\nalpham$ such that every positive integer $n \ge \nalpham$ has a partition with distinct parts, all larger than or equal to $m$, and such that the sum of reciprocals is equal to $\alpha$. No attempt was made to estimate the quantity $\nalpham$, however. With $\nalpha := \nalphaone$, in this paper we provide near-optimal upper bounds on $\nalpha$ and $\nalpham$, as well as bounds on the cardinality of the set $\{\alpha : \nalpha \le n\}$.

\section{Introduction and overview of results}
Given a set of positive integers $\{a_1, \ldots, a_r\}$, we say that this is an $\alpha$-partition of $n$, if all $a_i$ are distinct, $a_1 + \ldots + a_r = n$, and $\frac{1}{a_1} + \ldots + \frac{1}{a_r} = \alpha$. For positive integers $m$ and $M$, we say that a partition $A$ is \mlarge if all elements in $A$ are larger than or equal to $m$, and we call $A$ \mfree if none of the elements in $A$ are divisible by $M$. We furthermore define $\nalpham$ (which we shorten to $\nalpha$ if $m = 1$) to be the smallest positive integer such that, for all $n \ge \nalpham$, an \mlarge $\alpha$-partition of $n$ exists. Finally, define $\setB{n}$ as the set of all positive rationals $\alpha$ for which an $\alpha$-partition of $n$ exists, and define $\setA{n} \subseteq \setB{n}$ as the set of all $\alpha$ with $\nalpha \le n$. \\

In \mbox{\cite{gr}} Graham showed that $\n{1}$ is equal to $78$, and more generally that $\nalpha$ exists for all positive rationals $\alpha$. The proof, however, relied on a non-constructive (though beautifully general) result from \mbox{\cite{gr2}}. No bound on $\nalpha$ could therefore be extracted, and nothing beyond its existence seems to be known. To rectify this situation, in Section \ref{upperone} we will provide the first known upper bound on $\nalpha$. More specifically, we will show that for all $\epsilon > 0$ and all $\alpha = \frac{p}{q} \in (\epsilon, \epsilon^{-1})$ we have $\nalpha = o\left(q\log^3 q\right)$. We will moreover prove $\nalpha \neq o\left(q \log^2 q\right)$, so that we may conclude that our upper bound on $\nalpha$ is tight up to less than a single log-factor. In order to show this upper bound, we will rely on various results on unit fractions from the literature, most notably by Yokota \mbox{\cite{egypt1}}, Bloom \mbox{\cite{bloom}}, Liu and Sawhney \mbox{\cite{ls}}, and the author \mbox{\cite{wvd}}. \\

Moving on to $\nalpham$ for larger values of $m$, in Section \ref{uppertwo} we will make essential use of theorems by Croot \mbox{\cite{cr}} to prove the asymptotically optimal upper bound $\nalpham = \big(\frac{1}{2} + o_{\alpha}(1)\big)(e^{2\alpha}-1)m^2$, for fixed $\alpha > 0$. \\

Finally, we will look at the cardinality of the set $\setA{n}$. As mentioned, Graham \mbox{\cite{gr}} showed that the infinite union $\bigcup_{n \ge 1} \setA{n}$ is equal to the set of positive rationals, with $1 \in \setA{78}$. In \mbox{\cite{wvd}} the author showed that $\setA{65}$ is empty, while $\setA{100}$ contains exactly $4314$ rationals, and it is then natural to wonder about the growth rate of $|\setA{n}|$ as a function of $n$. In Section \ref{growth} we will combine results on \mfree partitions with known bounds on partition functions, in order to show $|\setA{n}| = e^{n^{1/2 + o(1)}}$. \\

As for our nomenclature and notation, we say that a positive integer $n$ is $x$-smooth if the largest prime divisor of $n$ is at most $x$, while $n$ is $x$-powersmooth if the largest prime power divisor of $n$ is at most $x$. Whenever we consider logarithms, we always interpret $\log(x) = 1$ if $x < e$. And as in \mbox{\cite{wvd}}, if $A = \{a_1, \ldots, a_r\}$ is a set of positive integers and $m \in \mathbb{N}$, then $\sum A$ denotes the sum $a_1 + \ldots + a_r$, $\sum A^{-1}$ is defined as the sum $\frac{1}{a_1} + \ldots + \frac{1}{a_r}$, and we write $mA$ for the set $\{ma_1, \ldots, ma_r\}$.

\section[First bound]{Bounds on $\nalpha$} \label{upperone}
The main goal of this section is to provide an upper bound on $\nalpha$ for any rational $\alpha > 0$. Afterwards we will show that this upper bound is less than a logarithmic factor away from optimal, if $\alpha$ is bounded away from $0$.

\begin{finalbound} \label{nalphabound}
For every $\epsilon > 0$ there exists a constant $c$ such that for all positive rationals $\alpha = \frac{p}{q}$ we have $\nalpha < \frac{c q \log^3 q}{\min(\alpha, 1) \log \log q} + c (e + \epsilon)^{2\alpha}$.
\end{finalbound}

In order to prove this theorem, we need some preliminary lemmata.

\begin{sninftylemma} \label{snifjes}
For a positive integer $M$, let $\alpha$, $\beta$ and $\gamma$ be positive rationals with $\alpha = \frac{\beta}{M} + \gamma$, and such that, for some $N \in \mathbb{N}$, \mfree $\gamma$-partitions exist for all integers larger than or equal to $N$. Then $\beta \in \setB{n_0}$ implies $\alpha \in \setA{Mn_0 + N}$. 
\end{sninftylemma}

\begin{proof}
Let $n$ be any integer larger than or equal to $N$, and assume that $\beta$ is an element of $\setB{n_0}$ for some $n_0 \in \mathbb{N}$. We may then assume the existence of a $\beta$-partition $B$ of $n_0$ and an \mfree $\gamma$-partition $C$ of $n$. Then we claim that $A = MB \cup C$ is an $\alpha$-partition of $Mn_0 + n$. Indeed, the sets $MB$ and $C$ are disjoint as $C$ is \mfree, the sum $\sum A$ is equal to $\sum MB + \sum C = Mn_0 + n$, and $\sum A^{-1} = \sum (MB)^{-1} + \sum C^{-1} = \frac{\beta}{M} + \gamma = \alpha$. Since this works for all $n \ge N$, we obtain $\alpha \in \setA{Mn_0 + N}$, or $\nalpha \le Mn_0 + N$.
\end{proof}

\begin{tenenbaumyokota} \label{tenyo}
For every positive integer $M$ there exists a constant $c_1$ such that for every rational $\beta \in (0, M]$ with denominator at most $t$, a positive integer $n_0 < \frac{c_1 t \log^3 t}{\log \log t}$ exists for which $\beta \in \setB{n_0}$. 
\end{tenenbaumyokota}

\begin{proof}
Let $M \in \mathbb{N}$ be fixed, let $C$ be as in Theorem $3$ from \mbox{\cite{bloom}}, and let $N > e^{2M(M+1)}$ be a large enough integer such that $\frac{C \log \log \log N}{\log \log N} < \frac{1}{M+1}$. By applying Theorem II.$2$ from \mbox{\cite{egypt1}} to $\beta_0 := \beta - \left \lfloor \beta \right \rfloor \in [0, 1)$, we first obtain a $\beta_0$-partition $B_0$ with $\sum B_0 < \frac{c_2 t \log^3 t}{\log \log t}$, for some absolute constant $c_2$. This already finishes the proof if $\beta \le 1$ by taking $c_1 := c_2$, so we may assume $1 < \beta \le M$. In this case we set $c_1 := c_2 + N^2$ and define $A := \{n : n \le N \text{ and } n \notin B_0\}$. \\

We now partition $A$ into a disjoint union $A_1 \cup A_2 \cup \ldots \cup A_{ \left \lfloor \beta \right \rfloor}$ in the following way: add the next element from $A$ to the set $A_i$ for which $\sum A_i^{-1}$ is currently smallest. This greedy procedure makes it so that $\left|\sum A_i^{-1} - \sum A_j^{-1}\right| \le 1$ for all $i$ and $j$. In particular, by applying $\sum B_0^{-1} = \beta_0 \le \beta \le M$, $N > e^{2M(M+1)}$ and $\frac{C \log \log \log N}{\log \log N} < \frac{1}{M+1}$, we get the following lower bound on $\sum A_i^{-1}$ for all $i$:
\begin{align*}
\sum A_i^{-1} &\ge \frac{1}{\left \lfloor \beta \right \rfloor}\sum A^{-1} - 1 \\
&\ge \frac{1}{M}\left(\frac{1}{1} + \frac{1}{2} +\ldots + \frac{1}{N} - \beta_0 \right) - 1 \\
&> \frac{\log N}{M} - 2 \\
&> \frac{\log N}{M+1} \\
&> \frac{C \log N \log \log \log N}{\log \log N}
\end{align*}

It then follows from \mbox{\cite{bloom}} that subsets $B_i \subseteq A_i$ exist with $\sum B_i^{-1} = 1$ for all $i \ge 1$, and note that the sets $B_i$ (including $B_0$) are all pairwise disjoint. In particular, with $n_0 = \sum B_0 + \sum B_1 + \ldots + \sum B_{ \left \lfloor \beta \right \rfloor}$ we see that $B_0 \cup B_1 \cup \ldots \cup B_{\left \lfloor \beta \right \rfloor}$ is now a $\beta$-partition of $n_0$, while $n_0 \le \sum B_0 + \sum A < \frac{c_2 t \log^3 t}{\log \log t} + N^2 \le \frac{c_1 t \log^3 t}{\log \log t}$. 
\end{proof}

\begin{halfandone} \label{halfandone}
For every $n \ge 531$, a $1$-partition $C_1$ of $n$ exists such that all elements in $C_1$ are $3$-smooth. In particular, this $1$-partition is \free{$M$} for all positive integers $M$ with $M \equiv 0 \pmod{5}$.
\end{halfandone}

\begin{proof}
This is Proposition $1$ in \mbox{\cite{wvd}}.
\end{proof}

\begin{ronald} \label{graham}
For every $\epsilon > 0$ there exists a positive integer $M \equiv 0 \pmod{5}$ and a constant $c_3$, such that for all positive integers $\gamma$, an \mfree $(\gamma-1)$-partition $C_2$ exists with $\sum C_2 < c_3(e + \epsilon)^{2\gamma}$ and such that no elements in $C_2$ are $3$-smooth. 
\end{ronald}

\begin{proof}
Given $\epsilon > 0$, we may assume without loss of generality $\epsilon < \frac{1}{5}$. First, let $M$ be any multiple of $5$ with $M > \frac{32}{\epsilon}$. Secondly, let $N$ be a large enough integer so that the set of $3$-smooth integers in $[1, n]$ has cardinality less than $\frac{\epsilon n}{32} - 1$ for all $n \ge N$. Further assume that $N$ is also large enough so that for all positive integers $n \ge N$ and all sets $A \subseteq \{1, 2, \ldots, n\}$ with $|A| > (1 - \frac{1}{e} + \frac{\epsilon}{16})n$, a subset $B \subseteq A$ exists with $\sum B^{-1} = 1$. That such an $N$ exists is proven as Theorem $1.3$ in \mbox{\cite{ls}}. For a non-negative integer $i$ define $N_i := N(e + \epsilon)^i$, and let $A_i$ be the set of integers in the interval $[N_i, N_{i+1})$ that are not  $3$-smooth and not divisible by $M$. We then get the following bound on the size of $A_i$:
\begin{align*}
|A_i| &> (N_{i+1} - N_i - 1) - \left(\frac{\epsilon N_{i+1}}{32} - 1\right) - \frac{\epsilon N_{i+1}}{32} \\
&= \left(1 - \frac{1}{e + \epsilon} - \frac{\epsilon}{16}\right)N_{i+1} \\
&= \left(1 - \frac{1}{e} + \frac{\epsilon}{e(e + \epsilon)} - \frac{\epsilon}{16}\right)N_{i+1} \\
&> \left(1 - \frac{1}{e} + \frac{\epsilon}{16}\right)N_{i+1}
\end{align*}

Here in the last line we used the assumption $\epsilon < \frac{1}{5}$. By this lower bound on $|A_i|$ we in particular get from \mbox{\cite{ls}} that for every $i$ there exists a subset $B_i \subseteq A_i$ with $\sum B_i^{-1} = 1$, and note that these sets $B_i$ are all pairwise disjoint. We now define $C_2 := B_0 \cup B_1 \cup \ldots \cup B_{\gamma - 2}$, with $\sum C_2 < N_{\gamma - 1}^2 < c_3(e + \epsilon)^{2\gamma}$ by setting $c_3 := N^2$. By construction $C_2$ is indeed an \mfree $(\gamma-1)$-partition without any $3$-smooth elements.
\end{proof}

\begin{allhalfs} \label{allhalfs}
For every $\epsilon > 0$ there exists a positive integer $M$ and a constant $c_4$, such that with $\gamma$ any positive integer and $N = \left \lfloor c_4(e + \epsilon)^{2\gamma} \right \rfloor$, an \mfree $\gamma$-partition exists for all $n \ge N$.
\end{allhalfs}

\begin{proof}
For given $\epsilon > 0$ and $\gamma \in \mathbb{N}$, let $M$, $c_3$ and $C_2$ be as in Lemma \ref{graham}, and define $c_4 := 531 + c_3$ and $N := \left \lfloor c_4(e + \epsilon)^{2\gamma} \right \rfloor$. For any positive integer $n \ge N \ge 531 + \sum C_2$, let $C_1$ be a $1$-partition of $n - \sum C_2$ as in Lemma \ref{halfandone}. Since $C_1 \cap C_2 = \emptyset$ by construction of $C_1$ and $C_2$, while $M$ does not divide any of the elements in $C_1 \cup C_2$, we deduce that $C_1 \cup C_2$ is an \mfree $\gamma$-partition of $n$. 
\end{proof}

\begin{arbsmall} \label{arbsmall}
For every integer $k \ge 2$ there exists an $N < 106 \cdot 4^k$ such that a \free{$5$} $\frac{4}{3^{k}}$-partition $C$ of $n$ exists for every $n \ge N$.
\end{arbsmall}

\begin{proof}
This is (part of) Lemma $4$ in \mbox{\cite{wvd}}.
\end{proof}

We are now ready to prove Theorem \ref{nalphabound}.

\begin{proof}[Proof of Theorem \ref{nalphabound}]
Let $\epsilon > 0$ and $\alpha = \frac{p}{q} > 0$ be given, and let $M$ and $c_4$ be as in Lemma \ref{allhalfs}. First we assume $\alpha > 1$. \\

In this case, let $\gamma$ be the largest positive integer with $\gamma < \alpha$, and define $\beta := M(\alpha - \gamma)$. Since $\beta \in (0, M]$ we may apply Lemma \ref{tenyo} with $t \le q$ to deduce $\beta \in \setB{n_0}$ for some $n_0 < \frac{c_1q \log^3 q}{\log \log q}$. And now we apply Lemma \ref{snifjes} with $N$ as in Lemma \ref{allhalfs} to obtain $\alpha = \frac{\beta}{M} + \gamma \in \setA{Mn_0 + N}$ with $\nalpha \le Mn_0 + N < \frac{Mc_1q \log^3 q}{\log \log q} + c_4(e + \epsilon)^{2\alpha}$. Taking $c = \max(Mc_1, c_4)$ finishes the proof for $\alpha > 1$. \\

For $\alpha \le 1$, let $k \ge 2$ be such that $\frac{4}{3^{k}} < \alpha \le \frac{4}{3^{k-1}}$ and choose $\gamma := \frac{4}{3^{k}}$. We then set $\beta := 5(\alpha - \gamma) \in (0, 5]$ with denominator $t$, apply Lemma \ref{snifjes} with $M = 5$ and $N$ as in Lemma \ref{arbsmall}, to obtain the bound $\nalpha < \frac{5c_1t \log^3 t}{\log \log t} + N$. It therefore remains to estimate this latter quantity, starting with $N$. 
\begin{align*}
N &< 106 \cdot 4^k \\
&= 106 \cdot 12^{\log 4 / \hspace{-1pt} \log 3} \cdot \left(\frac{3^{k-1}}{4} \right)^{\log 4 / \hspace{-1pt} \log 3} \\
&< \frac{2500}{\alpha^{\log 4 / \hspace{-1pt} \log 3}} \\
&= \frac{2500}{\alpha} \cdot \left(\frac{q}{p}\right)^{\log 4 / \hspace{-1pt} \log 3 - 1} \\
&< \frac{2500q}{\alpha}
\end{align*}

Now recall that $t$ is the denominator of $\beta = 5(\alpha - \gamma)$. And since $\gamma = \frac{4}{3^{k}} < \alpha \le \frac{4}{3^{k-1}}$, we get $t \le 3^kq \le \frac{12q}{\alpha} \le 12q^2$, so that $\log^3 t \le \log^3 12q^2 < 200\log^3 q$ and $t \log^3 t < \frac{2400q \log^3 q}{\alpha}$. Combining this inequality with our upper bound on $N$ provides the required upper bound on $\nalpha$ by taking $c = 12000c_1 + 2500$.
\begin{align*}
\nalpha &< \frac{5c_1t \log^3 t}{\log \log t} + N \\
&< \frac{12000c_1 q \log^3 q}{\alpha \log \log q} + \frac{2500q}{\alpha}\\
&< \frac{cq \log^3 q}{\alpha \log \log q} \qedhere
\end{align*}
\end{proof}

The upper bound on $\nalpha$ from Theorem \ref{nalphabound} actually turns out to be fairly tight. On the one hand, note that for a positive integer $k \le \frac{1}{3}e^{\alpha}$, we get by the classical estimate on the harmonic series $\frac{1}{1} + \frac{1}{2} + \ldots + \frac{1}{k} < 1 + \log k < \alpha$. This implies that every $\alpha$-partition needs more than $\frac{1}{3}e^{\alpha}$ distinct terms, from which we deduce $\nalpha \ge 1 + 2 + \ldots + \left \lceil \frac{1}{3}e^{\alpha} \right \rceil > \frac{1}{18}e^{2\alpha}$. Up to the $\epsilon$-term, this shows the necessity of the second term in the upper bound on $\nalpha$. \\

On the other hand, if $q$ is prime and $\alpha = \frac{p}{q}$ is bounded away from zero, then we claim that the first term is also near-optimal; up to less than a single log-factor, in fact. The proof of this latter claim is inspired by the proof of Theorem II.1 in \mbox{\cite{egypt1}}, which is itself borrowed from the proof of Theorem $1$ in \mbox{\cite{bler}}.

\begin{primesarehard}
Let $I$ be any non-degenerate interval of non-negative real numbers. Then for every large enough prime $q$ there exists a positive rational $\alpha = \frac{p}{q} \in I$ for which $\alpha \notin \setB{n}$ for all $n \le 0.3 q \log^2 q$. In particular, $\nalpha > 0.3 q \log^2 q$.
\end{primesarehard}

\begin{proof}
With $Q(x)$ defined as the number of partitions of $x$ into distinct positive integers, by Corollary $2$ in \mbox{\cite[p. 5]{part}} we have the inequality $Q(x) < e^{c \sqrt{x}}$, where $c = \frac{\pi}{\sqrt{3}} < 1.82$. In particular, for all large enough $x$ we have $\sum_{i \le x} Q(i) \le xQ(x) < e^{1.82 \sqrt{x}}$. Now let $q$ be any large enough prime such that this latter inequality holds with $x := \left \lfloor 0.3 \log^2 q \right \rfloor$. Moreover assume that $q$ is large enough so that for some $p_0 \in \mathbb{N}$, the interval $\left[\frac{p_0}{q}, \frac{p_0 + q^{0.999}}{q}\right]$ is contained in $I$. Since $\sum_{i \le x} Q(i) < e^{1.82 \sqrt{x}} < q^{0.999}$ this implies that there is a $p \in [p_0, p_0 + q^{0.999}]$ such that $p \not \equiv \sum B'^{-1} \pmod{q}$ for any set $B'$ with $\sum B' \le x$. \\

With $\alpha := \frac{p}{q}$, assume that $A$ is an $\alpha$-partition of $n$ for some $n < q^2$, and write $A$ as a disjoint union $B \cup C$ such that all $b \in B$ are divisible by $q$, while all $c \in C$ are not divisible by $q$. \\

From the inequalities $\sum B \le n < q^2$, it follows that no elements in $B$ are divisible by $q^2$, so that with $B' := \{bq^{-1} : b \in B\}$ we have $p \equiv \sum B'^{-1} \pmod{q}$. By construction of $p$, this implies $n \ge q \sum B' \ge q(x+1) > 0.3 q \log^2 q$.
\end{proof}

\section[Second bound]{Bounds on $\nalpham$} \label{uppertwo}
In this section we will find an asymptotic formula for $\nalpham$, when $m$ is large. 

\begin{explicitgraham} \label{superfinal}
For every fixed $\alpha > 0$ we have $\nalpham = \big(\frac{1}{2} + o_{\alpha}(1)\big)(e^{2\alpha} - 1)m^2$.
\end{explicitgraham}

\begin{proof}
Assume that $\alpha > 0$ and $\epsilon_0 > 0$ with $\epsilon_0 < \min(1 - e^{-\alpha}, \frac{1}{3})$ are arbitrary but fixed, let $m$ be sufficiently large in terms of $\alpha$ and $\epsilon_0$, and define $I_1, I_2$ to be the sets of integers within the intervals $(m, me^{\alpha})$ and $\big(me^{\alpha}, (1 + \epsilon_0)me^{\alpha}\big)$ respectively. Then $\sum I_1^{-1} = \alpha + o(1)$, so that a moment's thought gives $\nalpham \ge \big(1 - o(1)\big)\sum I_1 = \big(\frac{1}{2} - o(1)\big)(e^{2\alpha} - 1)m^2$. It therefore suffices to show a corresponding upper bound as well. For this, we will make crucial use of the following two results of Croot.

\begin{cr1} \label{cr1}
There exists a set $C_1 \subset I_1$ such that with $\beta := \alpha - \sum C_1^{-1}$ we have $\beta = \big(3\alpha + o(1)\big)\frac{\log \log m}{\log m}$ and $\beta$ can be written as a fraction whose denominator is $m^{1/5}$-powersmooth.
\end{cr1}

\begin{proof}
This is (a special case of) Proposition $1$ in \mbox{\cite{cr}}, with $\epsilon = \frac{1}{20}$ and $c = e^{\alpha}$.
\end{proof}

For the rest of the proof, we let $\beta$ be as in Lemma \ref{cr1}.

\begin{cr2} \label{cr2}
For every fraction $\frac{s}{t}$ with $s, t$ coprime, $\frac{1}{2}\beta < \frac{s}{t} \le \beta$ and $t$ $m^{1/5}$-powersmooth, there exists a set $C_2 \subset I_2$ such that $\sum C_2^{-1} = \frac{s}{t}$.
\end{cr2}

\begin{proof}
This is (a special case of) Proposition $2$ in \mbox{\cite{cr}}, with $\epsilon = \frac{1}{20}$ and the function $f$, for example, equal to $f(x) = \big(\alpha + o(1)\big)\log \log x$.
\end{proof}

To be able to apply these results, we furthermore need a result on the existence of powersmooth integers.

\begin{smoothinarith} \label{smoothinarith}
For every $\epsilon_0$ with $0 < \epsilon_0 < \min(1 - e^{-\alpha}, \frac{1}{3})$ there exists a $\delta > 0$, such that for any residue class $a \pmod{210}$ and all large enough $x$, there are at least $\delta x$ integers contained in the interval $\left(x, (1 + \frac{\epsilon_0}{12})x\right)$ which are all $x^{1/13}$-powersmooth and congruent to $a \pmod{210}$.
\end{smoothinarith}

\begin{proof}
See for example Theorem $1$ in \cite{smooth}. This theorem implies in particular that the $x^{1/13}$-smooth integers have positive density in every residue class. It is therefore sufficient to show that almost all $x^{1/13}$-smooth integers are actually $x^{1/13}$-powersmooth as well. To see this, note that for every prime $p$ there are less than $x^{12/13}$ integers below $x$ which are divisible by a power of $p$ larger than $x^{1/13}$. Summing this over all $p \le x^{1/13}$ gives us the upper bound $\pi(x^{1/13}) x^{12/13} = o(x)$ for the number of $x^{1/13}$-smooth integers which are not $x^{1/13}$-powersmooth.
\end{proof}

Since asymptotically almost all integers in $I_1$ are in $C_1$, by applying Lemma \ref{smoothinarith} with $x = \big(1 - \epsilon_0 \big) me^{\alpha} > m$ we deduce the existence of an integer $k \in C_1$ such that $k > \big(1 - \epsilon_0 \big) me^{\alpha}$ is coprime to $210 = 2 \cdot 3 \cdot 5 \cdot 7$ and $k+1$ is $m^{1/12}$-powersmooth. Now assume that a set $D = \{d_1, d_2, \ldots, d_l \}$ of distinct positive integers exists, with the following seven properties: 

\begin{enumerate}
	\item $\sum  D^{-1}= o\left(\frac{\log \log m}{\log m}\right)$.
	\item All $d_i \in D$ are $m^{1/6}$-powersmooth. 
	\item For every residue class $a \pmod{k}$ there exists a subset $D_2 \subseteq D$ with $\sum D_2 \equiv a \pmod{k}$.
	\item All $d_i \in D$ are larger than $k$. 
	\item None of the $d_i \in D$ are divisible by $k$. 
	\item For all $d, d' \in D$ and $x, x' \in \{20, 21, 28, 30\}$, the equality $dx = d'x'$ implies $d = d'$ and $x = x'$.
	\item $\sum D = O(\epsilon_0 m^2)$.
\end{enumerate}

With the (simplified) fraction $\frac{s}{t}$ defined as $\beta - \sum \left(12D\right)^{-1}$, we see by the first two properties of $D$ that Lemma \ref{cr2} applies to $\frac{s}{t}$. Let $C_2 \subset I_2$ therefore be such that $\sum C_2^{-1} = \frac{s}{t}$. \\

With $X$ defined as $76k + \sum I_1 + \sum I_2 + 50 \sum D$, let $n$ be any positive integer larger than $X$. Define $Y := \sum C_1 + \sum C_2 + 49 \sum D$ and let $D_2 \subseteq D$ be such that $\sum D_2 \equiv n - Y \pmod{k}$. Further define $D_1 := D \setminus D_2$ and $n' := \frac{n - Y + k - \sum D_2}{k}$, and note that $n'$ is an integer. Now let $B$ be a $1$-partition of $n'$, which exists, by the fact that $n'$ is larger than $77$; 
\begin{align*}
n'&= \frac{1}{k} \big(n - Y + k - \sum D_2\big) \\
&> \frac{1}{k} \big(X - Y + k - \sum D\big) \\
&\ge 77
\end{align*}

We now claim that $A$ is an \mlarge $\alpha$-partition of $n$, with $A$ defined as follows:
\begin{align*}
A &:= kB \cup (C_1 \setminus \{k\}) \cup C_2 \cup 21D_1 \cup 28D_1 \cup 20D_2 \cup 30D_2
\end{align*} 

To prove this, first we note that the fourth property of $D$ and the definitions of $k, C_1$ and $C_2$ imply that this partition is \mlarge. Secondly, let us check that these sets are all pairwise disjoint. \\

All elements in $(C_1 \setminus \{k\}) \cup C_2$ are smaller than $(1 + \epsilon_0)me^{\alpha}$. On the other hand, all other elements are larger than or equal to $2k$ by the fourth property of $D$, while $2k > \big(2 - 2\epsilon_0)me^{\alpha} > (1 + \epsilon_0)me^{\alpha}$. If $kb = dx$ for some $b \in B$, $d \in D$ and $x \in \{20, 21, 28, 30\}$, then $k$ needs to divide $d$, as $\gcd(k, x) = 1$ by definition of $k$. This is impossible however, by the fifth property of $D$. Finally, the sets $21D_1, 28D_1, 20D_2, 30D_2$ are all pairwise disjoint by the sixth property of $D$. \\

To see that $\sum A = n$, we calculate as follows.
\begin{align*}
\sum A &= \sum kB + \sum (C_1 \setminus \{k\}) + \sum C_2 + \sum 21D_1 \\
&\hspace{43pt}+ \sum 28D_1 + \sum 20D_2 + \sum 30D_2 \\
&= \sum kB + \sum C_1 - k + \sum C_2 + 49\sum D + \sum D_2 \\
&= kn' - k + Y + \sum D_2 \\
&= n
\end{align*}

Finally, to see $\sum A^{-1} = \alpha$, we use the equalities $\frac{1}{21} + \frac{1}{28} = \frac{1}{20} + \frac{1}{30} = \frac{1}{12}$ and $\sum (kB)^{-1} = \frac{1}{k}$. We then get $\sum A^{-1} = \sum C_1^{-1} + \sum C_2^{-1} + \sum (12D)^{-1}$. And this latter sum is equal to $\alpha - \beta + \frac{s}{t} + \sum (12D)^{-1} = \alpha$. \\

Now, the above works for all $n > X$, where $X = 76k + \sum I_1 + \sum I_2 + 50 \sum D$. Since $76k = o(m^2)$, $\sum I_2 = O(\epsilon_0 m^2)$ and (by the seventh property of $D$) $50\sum D = O(\epsilon_0 m^2)$, we get $X = \big(1 + o(1)\big)\sum I_1 = \big(\frac{1}{2} + o(1)\big)(e^{2\alpha} - 1)m^2$ for large enough $m$ by letting $\epsilon_0$ go to $0$. The only thing left to do is to create the set $D$ and prove that it has the required properties. \\

Let $e_1$ be equal to $1$ and, for $i \ge 2$, define $e_i$ to be the largest $m^{1/12}$-powersmooth integer smaller than or equal to $2e_{i-1}$. Let $l_1$ be such that $e_{l_1} \in [\frac{1}{2}\epsilon_0 k, \epsilon_0 k)$, and write $e_{l_1} = \epsilon_1 k$ for some $\epsilon_1$ with $\frac{1}{2} \epsilon_0 \le \epsilon_1 < \epsilon_0$. For $1 \le i \le l_1$ we define $d_i := e_i(k+1)$. Furthermore, define $l_2 := \left \lfloor \frac{-\log \epsilon_1}{\log \frac{4}{3}} \right \rfloor$ and, for $1 \le i \le l_2$, define the interval $J_i := (k + \left(\frac{4}{3} \right)^{i-1} \epsilon_1 k, k + \left(\frac{4}{3} \right)^i \epsilon_1 k) \subset (k, 2k)$. We then let $d_{l_1+i}$ be any $m^{1/6}$-powersmooth integer in the interval $J_i$ which is coprime to $210$. With $l := l_1 + l_2$, let us then prove that the set $D := \{d_1, \ldots, d_{l}\}$ exists and has the seven properties we need. 

\begin{existenceofd} \label{weexist}
Let $d_i$ and $e_i$ be defined as above. Then $d_i$ exists for all $i$ with $1 \le i \le l$, and $e_i > \frac{3}{2}e_{i-1}$ for all $i$ with $1 \le i \le l_1$.
\end{existenceofd}

\begin{proof}
It is clear that $e_i$ (and therefore $d_i$) exists for all $i \le l_1$ with $2^i \le m^{1/12}$, as we then simply have $e_i = 2^i$. When $i \le l_1$ is such that $2^i > m^{1/12}$, we apply Lemma \ref{smoothinarith} with $x = \frac{3}{2}e_{i-1} > \frac{3}{4}m^{1/12}$, which gives $e_i > \frac{3}{2}e_{i-1}$ for all $i \le l_1$. As for $d_{l_1+i}$ with $1 \le i \le l_2$, we apply Lemma \ref{smoothinarith} with $x = k + \left(\frac{4}{3} \right)^{i-1} \epsilon_1 k$. It is then sufficient to show $\left(x, (1 + \frac{e_0}{12})x\right) \subset J_i$. Or, equivalently, that the ratio of the endpoints of $J_i$ is at least $1 + \frac{\epsilon_0}{12}$.
\begin{align*}
\frac{k + \left(\frac{4}{3} \right)^i \epsilon_1 k}{k + \left(\frac{4}{3} \right)^{i-1} \epsilon_1 k} &= \frac{k + \left(\frac{4}{3} \right)^{i-1} \epsilon_1 k + \left(\left(\frac{4}{3} \right)^{i} - \left(\frac{4}{3} \right)^{i-1}\right)\epsilon_1 k }{k + \left(\frac{4}{3} \right)^{i-1} \epsilon_1 k} \\
&= 1 + \frac{\left(\frac{4}{3} \right)^{i-1}\left(\frac{4}{3} - 1\right)\epsilon_1 k }{k + \left(\frac{4}{3} \right)^{i-1} \epsilon_1 k} \\
&> 1 + \frac{\frac{1}{3}\epsilon_1k}{2k} \\
&\ge 1 + \frac{\epsilon_0}{12} \qedhere
\end{align*}
\end{proof}

Now that we know that all $d_i$ exist, it is clear from their definition that they are larger than $k$, not divisible by $k$ and, since $k+1$ is by definition $m^{1/12}$-powersmooth, $m^{1/6}$-powersmooth. To prove the third property that the subset sums of $D$ cover all residue classes modulo $k$, we need a well-known lemma.

\begin{completesequences} \label{complete}
If $X = \{x_1, \ldots, x_l\}$ is a set of positive integers with $1 = x_1 < \ldots < x_l$ and $x_{i} \le 2x_{i-1}$ for all $i$ with $2 \le i \le l$, then all integers in $\{1, \ldots, 2x_l-1\}$ can be written as the sum of distinct elements of $X$.
\end{completesequences}

\begin{proof}
This follows immediately by induction.
\end{proof}

For $1 \le i \le l$, define $x_i$ to be the unique positive integer smaller than or equal to $k$ with $x_i \equiv d_i \pmod{k}$, and let $X$ be the set $\{x_1, \ldots, x_l\}$. We then claim $x_1 = 1$, $x_i \le 2x_{i-1}$ for all $i$ with $2 \le i \le l$, and $x_l > \frac{1}{2}k$. From these claims we then conclude by Lemma \ref{complete} that all integers in $\{1, \ldots, k\}$ can be written as the sum of distinct elements of $X$, which implies the third property. The (in)equalities $x_1 = 1$ and $x_i \le 2x_{i-1}$ for all $i \le l_1 + 1$ follow directly from the definition of $d_i$ for $i \le l_1 + 1$. To see why $x_i \le 2x_{i-1}$ also holds for $i > l_1 + 1$, note that $x_{i} < \left(\frac{4}{3}\right)^i \epsilon_1 k = \frac{16}{9}\left(\frac{4}{3}\right)^{i-2} \epsilon_1 k < 2x_{i-1}$. Finally, $x_l > \left(\frac{4}{3} \right)^{l_2-1}\epsilon_1k > \left(\frac{4}{3} \right)^{\frac{-\log \epsilon_1}{\log \frac{4}{3}}-2}\epsilon_1k = \frac{9}{16}k > \frac{1}{2}k$. \\

By Lemma \ref{weexist} and the definition of $D$ we have $|D| = l = O(\log k)$. Since all elements in $D$ are larger than $k > m$, the first property on the size of $\sum D^{-1}$ follows. For the seventh property on the size of $\sum D$ we need another quick lemma.

\begin{sumofstuff} \label{sumofstuff}
If $X = \{x_1, \ldots, x_{l_1}\}$ is a set of positive integers with $1 = x_1 < \ldots < x_{l_1}$ and $x_i > \frac{3}{2}x_{i-1}$ for all $i$, then $\sum X < 3x_{l_1}$.
\end{sumofstuff}

\begin{proof}
This also follows immediately by induction.
\end{proof}

Applying Lemma \ref{sumofstuff} to the set $X = \{e_1, \ldots, e_{l_1}\}$ and we deduce $\sum X < 3e_{l_1} < 3\epsilon_0 k$, so that $\sum D = (k+1) \sum X + d_{l_1 + 1} + \ldots + d_{l} < 4\epsilon_0k^2 + 2l_2k = O(\epsilon_0m^2)$. \\

The final property we have to deal with is the sixth one; the non-existence of non-trivial solutions to $dx = d'x'$ with $d, d' \in D$ and $x, x' \in \{20, 21, 28, 30\}$. Since $d_i > \frac{3}{2}d_{i-1}$ for all $i \le l_1$ by Lemma \ref{weexist}, it is clear that this holds for all $d, d' \in \{d_1, \ldots, d_{l_1}\}$. On the other hand, if e.g. $d \in \{d_{l_1 + 1}, \ldots, d_l\}$, recall that we have $\gcd(d, 210) = 1$. We then see that $dx = d'x'$ implies $x = x'$, which in turn implies $d = d'$.
\end{proof}

\section[Growth rates]{Growth rates of $|\setA{n}|$ and $|\setB{n}|$} \label{growth}
Implicit in Lemma \ref{snifjes} is that \mfree partitions can be used to relate the sizes $\setA{n}$ and $\setB{n}$. In this final section we will first make this relation explicit, and then apply it in order to deduce lower and upper bounds on $|\setA{n}|$ and $|\setB{n}|$.

\begin{sninftybound} \label{sninfy}
We have the following inequality: $$ |\setB{n}| \le \min\big(|\setA{4n + 155}|, 158 \cdot |\setA{2n + 814}|\big) $$
\end{sninftybound}

\begin{proof}
By Theorem $1$ in \mbox{\cite{wvd}} we can apply Lemma \ref{snifjes} with $\gamma = 1$, $M = 4$ and $N = 155$. We then get a function $f$ from $\setB{n}$ to $\setA{4n + 155}$ which sends $\beta \in \setB{n}$ to $f(\beta) = \frac{\beta}{4} + 1 = \alpha \in \setA{4n + 155}$. Since $f$ is an injection, this proves the first part of the upper bound.\footnote{Lemma $1$ in \mbox{\cite{wvd}} explains why we cannot directly apply Lemma \ref{snifjes} here with $M < 4$.} \\

For the second part of the upper bound, we will define a function $f : \setB{n} \to \setA{2n + 814}$, such that for all $\alpha$ in the image of $f$, there are at most $158$ distinct $\beta \in \setB{n}$ with $f(\beta) = \alpha$. \\

For any $n \in \mathbb{N}$, let $n' \ge 2n + 814$ be a positive integer, and let the set $D = \{1, 2, 3, 4, 5, 6, 8, 18, 20, 24, 36\}$ be the union of the following sets $D_1, \ldots, D_8$:
\begin{align*}
D_{1} &:= \{1 \} \\ %&\sum D_i = 1
D_{2} &:= \{2, 4, 6, 18, 36 \} \\ %&\sum D_i = 66 
D_{3} &:= \{2, 3, 6 \} \\ %&\sum D_i = 11
D_{4} &:= \{2, 4, 8, 18, 24, 36 \} \\ %&\sum D_i = 92
D_{5} &:= \{2, 3, 8, 24 \} \\ %&\sum D_i = 37 
D_{6} &:= \{3, 4, 5, 6, 20 \} \\ %&\sum D_i = 38
D_{7} &:= \{2, 4, 5, 20 \} \\ %&S\sum D_i = 31 
D_{8} &:= \{3, 4, 5, 8, 20, 24 \} %&\sum D_i = 64
\end{align*}

For every $\beta \in \setB{n}$, choose a $\beta$-partition $B$ of $n$, and write $B$ as the disjoint union $B_1 \cup B_2$, where $b \in B_2$ if, and only if, $2b \in D$. Now define $\alpha = f(\beta) := \sum (2B_1)^{-1} + \frac{22}{15}$. For any $\alpha$ in the image of $f$, we then get that all $\beta \in \setB{n}$ with $f(\beta) = \alpha$ can be written as $2\alpha - \frac{44}{15} + \sum B_2^{-1}$, where $B_2$ is a set of integers with $2B_2 \subseteq D$. The cardinality of the inverse image $f^{-1}(\alpha)$ is therefore at most the number of distinct rationals that can be written as $\sum B_2^{-1}$ with $B_2 \subseteq \{1, 2, 3, 4, 9, 10, 12, 18 \}$. With a bruteforce search one can then check that there are exactly $158$ rationals that can be written in such a way. As we will show $\alpha \in \setB{n'}$ for all $\beta \in \setB{n}$, and $n'$ is an arbitrary integer larger than or equal to $2n + 814$, we deduce $\alpha \in \setA{2n+814}$, proving the second upper bound. \\

%these are the 158 different values: $\{0, 1, 2, 1/2, 1/3, 1/4, 1/6, 1/9, 1/10, 1/12, 1/18, 2/3, 3/2, 3/4, 3/5, 4/3, 4/9, 4/15, 5/3, 5/4, 5/6, 5/9, 5/12, 5/36, 7/3, 7/4, 7/6, 7/9, 7/12, 7/18, 7/20, 7/36, 7/45, 8/5, 8/9, 9/4, 10/9, 11/6, 11/9, 11/10, 11/12, 11/18, 11/36, 11/60, 13/6, 13/9, 13/12, 13/18, 13/30, 13/36, 14/9, 14/15, 16/9, 17/9, 17/12, 17/18, 17/20, 17/36, 19/12, 19/15, 19/18, 19/36, 19/90, 20/9, 21/10, 22/45, 23/12, 23/18, 23/30, 23/36, 25/12, 25/18, 25/36, 27/20, 29/15, 29/18, 29/36, 31/18, 31/36, 31/60, 32/45, 34/15, 35/18, 35/36, 37/20, 37/36, 37/45, 41/18, 41/36, 41/60, 43/30, 43/36, 43/180, 47/20, 47/36, 47/45, 49/36, 49/90, 52/45, 53/30, 53/36, 53/180, 55/36, 59/36, 59/90, 61/36, 61/60, 62/45, 65/36, 67/36, 67/45, 71/36, 71/60, 73/30, 73/36, 73/180, 77/36, 77/45, 79/36, 79/90, 82/45, 83/180, 89/90, 91/60, 92/45, 101/60, 103/180, 107/45, 109/90, 113/180, 119/90, 121/60, 131/60, 133/180, 139/90, 143/180, 149/90, 163/180, 169/90, 173/180, 179/90, 193/180, 203/180, 209/90, 223/180, 233/180, 253/180, 263/180, 283/180, 293/180, 313/180, 323/180, 343/180, 353/180, 373/180, 383/180, 403/180, 413/180]\}$ 
%PARI/GP code:
%s = Set([]);
%forsubset(18,x,subset = setintersect(Set(x),  Set([1, 2, 3, 4, 9, 10, 12, 18])); s=setunion(s,Set(sum(i=1,#subset,1/subset[i]))));
%print(#s);
%print(s);

One can check $\sum D_i \equiv i \pmod{8}$, while $\sum D_i^{-1} = 1$ for all $i$. For a given $\beta \in \setB{n}$, let $i$ be such that $m := n' - \sum 2B_1 - \sum D_i \equiv 1 \pmod{8}$. Since $n' \ge 2n + 814$, $\sum 2B_1 \le 2n$ and $\sum D_i \le 92$, we deduce $m \ge 729$, which implies by Corollary $4$ in \mbox{\cite{wvd}} that a \free{$2$} $\frac{7}{15}$-partition $C$ of $m$ exists, with $C \cap \{1, 3, 5\} = \emptyset$. We then claim that $A = 2B_1 \cup C \cup D_i$ is an $\alpha$-partition of $n'$. \\

To see this, let us first check that $2B_1$, $C$ and $D_i$ are pairwise disjoint. The sets $2B_1$ and $D_i \subset D$ are disjoint by construction of $2B_1 = \{2b \in B | 2b \notin D \}$, while the odd integers in $2B_1 \cup D_i$ are contained in $\{1, 3, 5\}$. On the other hand, $C$ does not contain any even integers, and does not contain $1, 3$ or $5$ either. \\

Since $\sum A = \sum 2B_1 + \sum C + \sum D_i = \sum 2B_1 + m + \sum D_i = n'$ and $\sum A^{-1} = \sum (2B_1)^{-1} + \sum C^{-1} + \sum D_i^{-1} = \sum (2B_1)^{-1} + \frac{7}{15} + 1 = \alpha$, the proof is complete.
\end{proof}

In the above proof we claimed that there are $158$ different rationals that can be written as $\sum B_2^{-1}$ with $B_2 \subseteq \{1, 2, 3, 4, 9, 10, 12, 18 \}$. But even without checking all possibilities either by hand or by computer, it is clear that there are at most $2^8 = 256$ such rationals. Using this more easily verifiable upper bound would change the constant $158$ that appears in the statement of Lemma \ref{sninfy} to $256$. It seems furthermore plausible that with a different value instead of $\frac{22}{15}$ or with different sets $D_i$, this upper bound can be somewhat improved. On the other hand, it is unclear whether we should expect $|\setB{n}| \le |\setA{n+c_0}|$ for some constant $c_0$, perhaps even $c_0 = 65$. \\

In any case, we may conclude that bounds on $|\setB{n}|$ imply corresponding bounds on $|\setA{n}|$. And this will bring us to our final result.

\begin{actualsninftybound} \label{actualsninfy}
With $c = \frac{\pi}{\sqrt{3}}$ we have the following bounds: $$e^{(c - o(1))\sqrt{\frac{n}{\log n}}} < |\setA{n}| \le |\setB{n}| < e^{c\sqrt{n}}$$
\end{actualsninftybound}

\begin{proof}
Recall that $Q(n)$ was defined as the number of partitions of $n$ into distinct positive integers, and $Q(n) < e^{c\sqrt{n}}$. Since we trivially have $|\setA{n}| \le |\setB{n}| \le Q(n)$, the final two inequalities follow. As for the lower bound, with $P(N)$ the number of partitions of $n$ into distinct primes, in \mbox{\cite{primepart}} it is proven that $\log P(n)$ is asymptotically equal to $c \sqrt{\frac{2n}{\log n}}$. Now $|\setB{n}| \ge P(n)$, as different partitions into distinct primes lead to different sums of reciprocals. Applying $|\setA{n}| \ge \frac{1}{158} |\setB{\left \lfloor \frac{1}{2}n \right \rfloor - 407}|$ from Lemma \ref{sninfy} and realizing that the constants $\frac{1}{158}$ and $407$ both vanish into the $o(1)$-term, finishes the proof.
\end{proof}

\end{document}